\documentclass[11pt]{article}

\usepackage[latin1]{inputenc}
\usepackage{amssymb,amsmath,amsthm}
\usepackage{mathrsfs}
\usepackage{graphicx}

  \newcommand{\lab}[1]{\label{#1}}

\if00  
\newcommand{\pic}
{
\includegraphics[height=10cm]{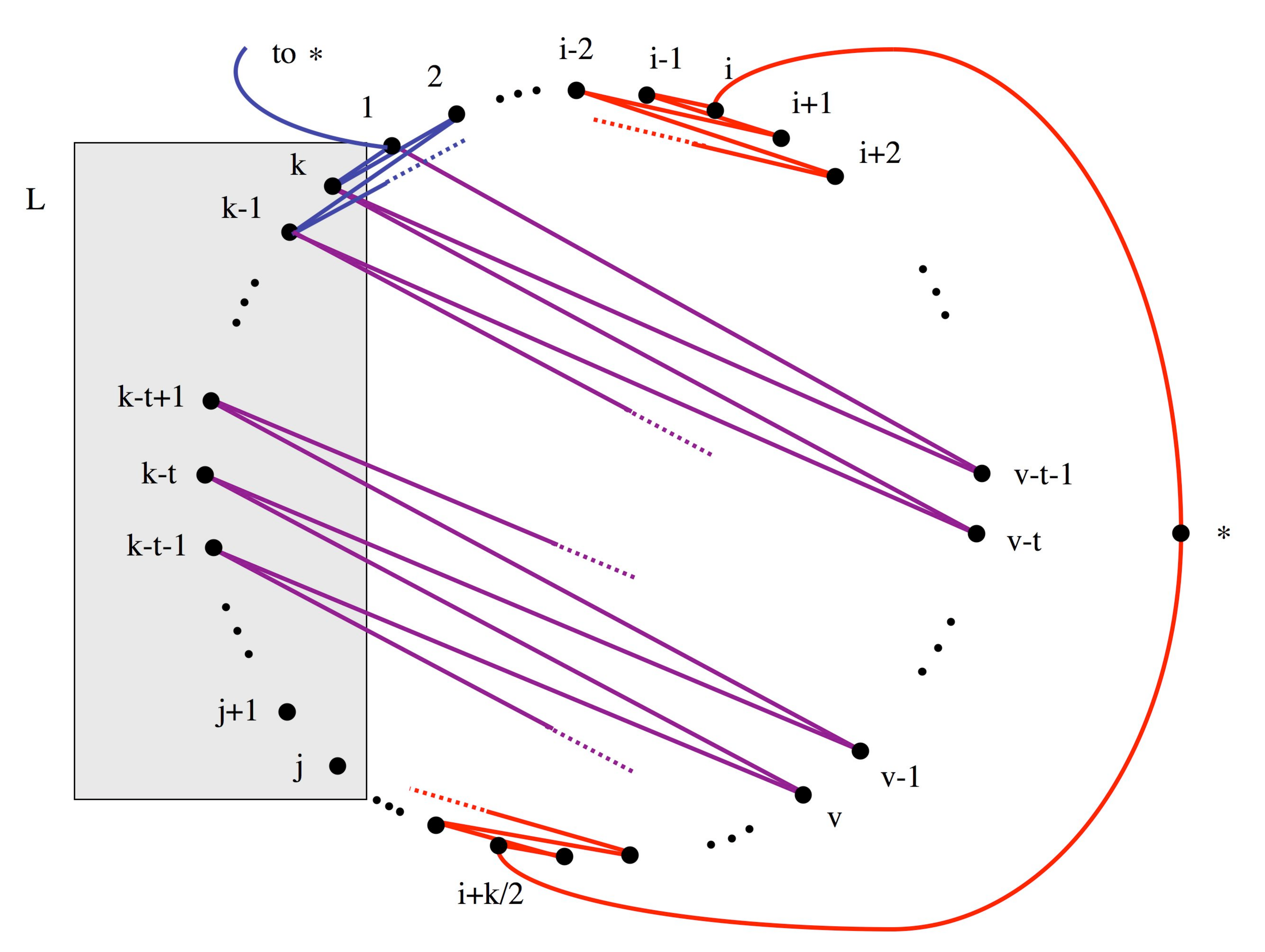}
\caption{\em  Parts of a hamiltonian decomposition. The colour $i$ is shown as red,   colour 1 is blue, and colour $v+k-t$ is purple.}
}
\else
\newcommand{\pic}
{
\includegraphics[height=10cm]{fact.jpg}
\caption{\em  Parts of a hamiltonian decomposition. The colour $i$ is shown as a dashed line,   colour 1 is dotted, and colour $v+k-t$ is solid.}
}
\fi

\if01
\renewcommand{\pic}
{
\caption{\em  Parts of a hamiltonian decomposition. The colour $i$ is shown as a dashed line,   colour 1 is dotted, and colour $v+k-t$ is solid.}
}
\fi

\theoremstyle{plain}
\newtheorem{thm}{Theorem}
\newtheorem{lem}[thm]{Lemma}

\theoremstyle{definition}
\theoremstyle{remark}

\newtheorem{case}{Case}

\newcommand{\real}{\ensuremath {\mathbb R} }

\newcommand{\mbf}[1] {\text{\boldmath$#1$}}
\newcommand{\remove}[1] {}

\newcommand{\st}{\;:\;}

\newcommand{\ex} {{\bf E}}
\newcommand{\pr} {{\bf Pr}}

\newcommand{\cB} {\ensuremath{\mathcal B}}
\newcommand{\cC} {\ensuremath{\mathcal C}}
\newcommand{\cD} {\ensuremath{\mathcal D}}

\newcommand{\cQ} {\ensuremath{\mathcal Q}}
\newcommand{\cR} {\ensuremath{\mathcal R}}
\newcommand{\cS} {\ensuremath{\mathcal S}}

\newcommand{\sG} {\ensuremath{\mathscr G}}
\newcommand{\bX} {\ensuremath{\mbf X}}

\newcommand{\US}{[0,1]^2}
\newcommand{\N}{\{1,\ldots,n\}}
\newcommand{\dist}{d}

\newcommand{\pip}{\pi_p}

\newcommand{\GXr}{\sG(\bX;r)}
\newcommand{\GXrl}{\sG(\bX;r_l)}
\newcommand{\GXrk}{\sG(\bX;r_k)}
\newcommand{\GXru}{\sG(\bX;r_u)}
\newcommand{\GXrp}{\sG(\bX;r')}

\newcommand{\GC}{\sG_\cC}

\newcommand{\GD}{\GC[\cD]}

\newcommand{\GB}{\GC[\cB]}

\newcommand{\BNE}{B_{^\nearrow}}
\newcommand{\BNW}{B_{^\nwarrow}}
\newcommand{\BSE}{B_{^\searrow}}
\newcommand{\BSW}{B_{^\swarrow}}

\title{Disjoint Hamilton cycles in the random geometric graph}
\author{
Xavier P\' erez-Gim\' enez\thanks{Partially supported by the Province of Ontario under the Post-Doctoral Fellowship (PDF) Program.}
\and
Nicholas Wormald\thanks{Supported by the  Canada Research Chairs Program and NSERC.}
}
\date{
{\small Department of Combinatorics and Optimization} \\
{\small University of Waterloo}\\
{\small Waterloo ON, Canada}\\
\smallskip
{\small {\tt \{xperez, nwormald\}@uwaterloo.ca }}
\bigskip
\\
July 26, 2009}

\begin{document}
\maketitle
\begin{abstract}
We prove a conjecture of Penrose about the standard random geometric graph process, in which $n$ vertices are placed at random on the unit square and edges are sequentially added in increasing order of lengths taken in the $\ell_p$ norm. We show that the first edge that makes the random geometric graph Hamiltonian is a.a.s.\ exactly the same one that gives $2$-connectivity. We also extend this result to arbitrary connectivity, by proving that the first edge in the process that creates a $k$-connected graph coincides a.a.s.\ with the first edge that causes the graph to contain $k/2$ pairwise edge-disjoint Hamilton cycles (for even $k$), or $(k-1)/2$ Hamilton cycles plus one perfect matching, all of them pairwise edge-disjoint (for odd $k$).
\end{abstract}
\section{Introduction}

Many authors have studied the evolution of the random geometric graph on $n$ labelled vertices placed independently and \emph{uniformly at random} (u.a.r.) on the unit square $\US$, in which edges are added in increasing order of length (see e.g.~\cite{Pe03}). Penrose~\cite{Pe99} proved that the first added edge that makes the graph have minimum degree $k$ is a.a.a.\ the first one that makes it $k$-connected. He also conjectured that on the evolution of the random geometric graph $2$-connectivity occurs a.a.s.\ precisely when the first Hamilton cycle is created. As a first step towards proving that conjecture, D\'\i az, Mitsche and the first author showed in~\cite{DiMiPe07} that the property of being Hamiltonian has a sharp threshold at $r\sim\sqrt{\log n/(\pip n)}$ (for a constant $\pi$ depending on the $\ell_p$-norm used), which coincides asymptotically with the threshold for $k$-connectivity for any constant $k$. In this paper we prove a result which, as a special case, establishes Penrose's conjecture.  

Independently and simultaneously with us obtaining our results, Penrose's conjecture was proved in the manuscripts by Balogh, Bollob\' as and Walters~\cite{BaBoWa} and by
Krivelevich and M\" uller~\cite{KrMu}. The arguments in these papers and ours are based on the ideas in~\cite{DiMiPe07}. 

Now consider the evolution of the random graph $G$ on $n$ labelled vertices, in which edges are added one by one. Bollob\' as and Frieze showed in~\cite{BoFr85} that \emph{asymptotically almost surely} as soon as $G$ has minimum degree $k$, it also contains $\lfloor k/2\rfloor$ edge-disjoint Hamilton cycles plus an additional edge disjoint perfect matching if $k$ is odd, where $k$ is any constant positive integer. (Here, asymptotically almost surely (a.a.s.) denotes with probability tending to $1$ as $n\to\infty$.) In particular, a.a.s.\ $G$ becomes $k$ connected as soon as the last vertex of degree less than $k$ disappears.

Our main result in this paper, conjectured by Krivelevich and M\" uller~\cite{KrMu}, is that the analogue of Bollob\' as and Frieze's result holds for the random geometric graph.  That is, we show that, in the evolution of the random geometric graph, a.a.s.\ as soon as the graph becomes $k$-connected, it immediately contains $\lfloor k/2\rfloor$ edge-disjoint Hamilton cycles plus one additional perfect matching if $k$ is odd. To ensure that the Hamilton cycles are edge-disjoint was a significant obstacle. To overcome it seems to require a deterministic result that apparently does not readily follow from other known results. This feature of our argument is not often found in proofs of properties of random structures.

Let $\bX=(X_1,\ldots,X_n)$ be a random vector, where each $X_i$ is a point in $\US$ chosen independently with uniform distribution.
Given $\bX$ and a radius $r=r(n)\ge0$, we define the {\em random geometric graph} $\GXr$ as follows: the vertex set of $\GXr$ is $\N$ and there is an edge joining $i$ and $j$ whenever $\Vert X_i-X_j\Vert_p\le r$. Here $\Vert\cdot\Vert_p$ denotes the standard $\ell_p$ norm, for some fixed $p\in[1,+\infty]$. Unless otherwise stated, all distances in $\US$ are measured according to the $\ell_p$ norm (i.e. $\dist(X,Y)=\Vert X-Y\Vert_p$). Let $\pip$ be the area of the unit $\ell_p$-ball (e.g.\ $\pi_2=\pi$, and $2\le\pip\le4$ for all $1\le p\le\infty$).

A continuous-time random graph process $\big(\GXr\big)_{0\le r\le1}$ is defined in a natural way, by first choosing the random set of points $\bX$ and then adding edges one by one as we increase the radius $r$ from $0$ to $\Vert(1,1)\Vert_p$.

\begin{thm}\lab{t:main} Consider the random graph process $\big(\GXr\big)_{0\le r\le1}$ for any $\ell_p$-normed metric on $\US$, and let $k$ be a fixed positive integer.
\begin{description}
\item{(i)} For even $k\ge 2$, a.a.s.\ the minimum radius $r$ at which the graph $\GXr$ is $k$-connected is equal to the minimum radius at which it has $k/2$ edge-disjoint Hamilton cycles.
\item{(ii)} For odd $k\ge 1$, a.a.s.\ the minimum radius $r$ at which the graph $\GXr$ is k-connected is equal to the minimum radius at which it has $(k-1)/2$ Hamilton cycles and one perfect matching, all of them pairwise edge-disjoint. (Here asymptotics are restricted to even $n$.) 
\end{description}
\end{thm}

To show that sets of pairwise edge-disjoint Hamilton cycles can be made to pass through certain ``bottlenecks" we will need a deterministic result about packing paths in graphs.
A {\em linear forest} is a forest all of whose components are paths. We use $d_G(v)$ and $N_G(v)$ to denote the degree and set of neighbours (respectively) of a vertex $v$ in  a graph $G$.
\begin{lem}\lab{l:colouring}
Assume $k\ge 1$, $1\le j\le k$  and $\ell=k-j+1$.
Let $G$ be a graph with vertex set $J\cup B$ with $|J|=j$ consisting of a clique on vertex set $J$  together with a bipartite graph $H$ with parts $J$ and $B$, such that
\begin{description}
 	\item{(i)} $d_H(v)\ge  \ell$ for  each $v\in J$, and there exists a special vertex we call the {\em apex} which has degree at least $\ell+1$   in $H$;
 	\item{(ii)} for each pair of distinct vertices $v,v'\in J$, $|N_H(\{v,v'\})\setminus\{v,v'\} |\ge\ell+1$.
\end{description}
Then $G$ contains a packing of $\lfloor k/2\rfloor $ pairwise edge-disjoint linear forests, and additionally a matching if $k$ is odd, which contain all edges in the clique with vertex set $J$, and such that each vertex in $J$ has degree 2 in each forest and (for odd $k$) degree 1 in the matching.
\end{lem}
Although we do not require it, we conjecture that Lemma~\ref{l:colouring} is still valid if one removes the requirement of existence of an apex vertex (i.e.\ simply require $d_H(v)\ge \ell$ for each $v\in J$ for the first condition of the lemma).

The next section contains the basic geometric definitions and probabilistic statements required in the argument, including proofs that several properties hold a.a.s. Next, we prove Lemma~\ref{l:colouring} in Section~\ref{s:colouring}. Finally, in Section~\ref{s:build}, we prove the main theorem, by supplying the required construction of  Hamilton cycles (and perfect matching)  in the random geometric graph deterministically, assuming the properties that were shown to hold a.a.s.


\section{Asymptotically almost sure properties}\lab{s:assprop}

Here, we define some properties of the random geometric graph that hold a.a.s.\ and that will turn out to be sufficient for our construction of disjoint Hamilton cycles.

Henceforth we assume that the points in $\bX$ are in general position---i.e.\ they are are all different, no three of them are collinear, and all distances between pairs of points are strictly different---since this holds with probability $1$.
\begin{lem}\lab{l:extradeg}
For any small enough constant $\eta>0$ and any $r$ such that $\pip nr^2=\log n+(k-1)\log\log n + O(1)$, the random geometric graph $\GXr$ a.a.s.\ satisfies the following property.
Every set $J$ of vertices of size $2\le|J|\le k$ in which each vertex has degree at least $k$ and such that $\max_{u,v\in J}\{\dist(X_u,X_v)\}\le \eta r$ contains some vertex of degree at least $k+1$.
\end{lem}
\begin{proof}
Suppose that some set $J\subset\N$ with $|J|=j$ causes the property to fail. Let $v_1$ and $v_2$ be two vertices of $J$ such that $X_{v_1}$ and $X_{v_2}$ realise the diameter of $\{X_v\st v\in J\}$. Enumerate the rest of $J$ as $v_3,\ldots, v_j$. We may assume that $v_1$ has degree exactly $k$; let  $v_{j+1},\ldots,v_{k+1}$ be its other neighbours. Letting $s$ denote the number of neighbours of $v_2$ which are not neighbours of $v_1$, we have an example of one of the following cofigurations.  A {\em bad} configuration is an ordered tuple $J$ of $k+s+1$ vertices $v_1,\ldots,v_{k+s+1}$ in $\GXr$ with the following properties: vertices $v_2,\ldots,v_{k+1}$ are the only $k$ neighbours of $v_1$ in $\GXr$; $\dist(v_1,v_2)\le \eta r$; $\dist(v_1,v_i)\le\dist(v_1,v_2)$ for $2< i\le j$; vertex $v_2$ has exactly $s$ neighbours, namely $v_{k+2},\ldots,v_{k+s+1}$, which are not neighbours of $v_1$. Let $T$ be the number of   bad configurations, for any fixed $s$. To prove the statement it suffices to show that $\ex T=o(1)$, regardless of the choice of $s$.

Assume we are given the position of $v_1$ and $\rho=\dist(v_1,v_2)$. The probability that $\dist(v_1,v_i)\le\rho$ for $2<i\le j$ is $(\pip\rho^2)^{j-2} = O(\rho^{2j-4})$. The probability that $\dist(v_1,v_i)\le r$ for $j+1\le i\le k+1$ is $(\pip r^2)^{k+1-j}=O(r^{2(k+1-j)})$. The probability that vertices $v_{k+2},\ldots,v_{k+s+1}$ are neighbours of $v_2$ but not $v_1$ is $(c\rho r)^s=O((\rho r)^s)$, where $|c-2|\le\epsilon$ for some small $\epsilon>0$ independent of $\rho$ (we can achieve that by taking $\eta$ small enough). The probability that the remaining $n-k-s-1$ vertices are not neighbours of $v_1$ or $v_2$ is
\[
(1-\pip r^2-c\rho r)^{n-k-s-1} \sim \exp(-(\pip r^2+c\rho r)n)
= O\left(\frac{e^{-c\rho rn}}{n\log^{k-1}n}\right).
\]
The probability density function of $\dist(v_1,v_2)$ is $2\pip\rho$. Putting all together,
\begin{align*}
\ex T &= O(n^{k+s+1}) \int_{0}^{\eta r} \rho^{2j-4} r^{2(k+1-j)} (\rho r)^s \frac{e^{-c\rho rn}}{n\log^{k-1}n} \rho d\rho
\\
&= O(1) \frac{(nr^2)^{k+2-2j}}{\log^{k-1}n} \int_{0}^{\eta cr^2n} x^{2j+s-3} e^{-x} dx
\\
&= O( \log^{3-2j}n) = o(1)\qedhere
\end{align*}
\end{proof}
For the following definitions, we fix $\delta>0$ to be a small enough constant and assume $r\to0$. We tessellate $\US$ into square \emph{cells} of side $\delta'r=\lceil(\delta r)^{-1}\rceil^{-1}$. (Note that $\delta'$ is not constant, but $\delta'\le\delta$ and $\delta'\to\delta$). Let $\cC$ be the set of cells, and let $\GC$ be an auxiliary graph with vertex set $\cC$ and with one edge connecting each pair of cells $c_1$ and $c_2$ iff all points in $c_1$ have distance at most $r$ from all points in $c_2$. Note that we shall use the term \emph{adjacent cells} to refer to cells which are adjacent vertices of the graph of cells $\GC$, while cells sharing a side boundary will be described as being \emph{topologically adjacent}. Let $\Delta$ be the maximum degree of $\GC$. By construction, $\Delta$ is a constant only depending on $\delta$ and the chosen $\ell_p$ norm.

We may assume that each point in $\bX$ lies strictly in the interior of a cell in the tesselation, since this happens with probability $1$.
Let $M$ be a large enough but constant positive integer (its choice will only depend on $\Delta$, thus on $\delta$, and also on $k$ and $\ell_p$). A cell in $\cC$ is {\em dense} if it contains at least $M$ points of the random set $\bX$, {\em sparse} if it contains at least one, but less than $M$, points in $\bX$, and {\em empty} if it has no points in $\bX$. Let $\cD\subseteq\cC$ be the set of dense cells. Note that $\cD\ne\emptyset$, since the total number of cells is $|\cC|=\Theta(n\log n)$, so at least one must contain $\Omega(\log n)$ points in $\bX$.

A set of cells is said to be \emph{connected} if it induces a connected subgraph of $\GC$.  (For $\delta$ small enough, this includes the situation where the union of cells is topologically connected.) The \emph{area} of a set of cells is simply the area of the corresponding union of cells. A set of cells \emph{touches} one side (or one corner) of $\US$ if it contains a cell which has some boundary on that side (or corner) of the unit square.
\begin{lem}\lab{l:P2}
For any constants $\delta>0$, $\alpha>0$ and $\lambda\in\real$ and for $r$ defined by $\pip nr^2=\log n+(k-1)\log\log n + \lambda$, the following statements hold a.a.s.
\begin{enumerate}
\item All connected sets of cells of area at least $(1+\alpha)\pip r^2$ contain some dense cell.
\item All connected sets of cells of area at least $(1+\alpha)\pip r^2/2$ touching some side of $\US$ contain some dense cell.
\item All cells contained inside a $5r\times5r$ square on each corner of $\US$ are dense.
\end{enumerate}
\end{lem}
\begin{proof}
Recall that the area of each cell is $\delta'^2r^2$. Then, in order to show the first statement in the lemma, it suffices to consider all connected sets of cells with exactly $s=\lceil(1+\alpha)\pip / \delta'^2\rceil=\Theta(1)$ cells. 
Let $\cS$ be such a set of cells. The probability that $\cS$ has no dense cell is at most
\begin{eqnarray*}
\sum_{i=0}^{(M-1)s} \binom{n}{i} (s\delta'^2r^2)^i (1-s\delta'^2r^2)^{n-i}
&=& O\left(e^{-(1+\alpha)\pip r^2 n}\right) \sum_{i=0}^{(M-1)s}(r^2n)^i\\
&=& O\left( n^{-(1+\alpha)} \log^c n \right),
\end{eqnarray*}
where $c=(M-1)s-(1+\alpha)(k-1)$ is constant. To conclude the first part of the proof,  multiply the probability above by the number $\Theta(1/r^2)=\Theta(n/\log n)$ of connected sets of $s$ cells.

By a completely analogous argument, if $\cS$ has area only $\lceil(1+\alpha)\pip / \delta'^2\rceil/2$ and touches some side of $\US$, the probability that it has no dense cell is $O(n^{-(1+\alpha)/2)}\log^{c'}n$, for some constant $c'$. However, the number of such sets is only $\Theta(\sqrt{n/\log n})$.

Finally, there is a bounded number of cells inside any of the $5r\times5r$ squares on the corners, and each individual cell is dense with probability $1-o(1)$.
\end{proof}

A set of cells is \emph{small} if it can be embedded in a $16\times16$ grid of cells, and it is \emph{large} otherwise.  Consider the subgraph $\GD$ of $\GC$ induced by dense cells, and let $\cD_0$ be the set of dense cells which are not in small components of $\GD$ (we shall see that $\cD_0$ forms a unique large component in $\GD$).
Most of the trouble in our argument comes from cells which are not adjacent to any dense cell in $\cD_0$, so let $\cB=\cC\setminus(\cD_0\cup N(\cD_0))$, and call the cells in $\cB$ \emph{bad} cells. Also, let us denote components of $\GB$ as \emph{bad} components.
Note that by construction all cells in $N(\cB)\setminus\cB$ must be sparse but adjacent to some cell in $\cD_0$, while $\cB$ itself may contain both sparse and dense cells.
%
%
\begin{lem}\lab{l:nolarge}
For a small enough constant $\delta>0$, any constant $\lambda\in\real$ and $r$ defined by $\pip nr^2=\log n+(k-1)\log\log n + \lambda$, the following holds a.a.s.
\begin{enumerate}
\item All components of $\GD$ are small except for one large component formed by precisely the cells in $\cD_0$.
\item $\GB$ has only small components.
\end{enumerate}
\end{lem}
\begin{proof}
First, we claim that the following statements are a.a.s.\ true. (Recall that ``connected'' is defined in terms of the graph $\GC$, not topological adjacency.)
\begin{enumerate}
\item
For any large connected set of cells $\cS$ such that $N(\cS)$ does not touch all four sides of $\US$, $N(\cS)\setminus\cS$ must contain some dense cell.
\item
For any pair of connected sets of cells $\cS_1$ and $\cS_2$ not adjacent to each other (i.e.\ $\cS_2\cap N(\cS_1)=\emptyset$) and such that both $N(\cS_1)$ and $N(\cS_2)$ touch all four sides of $\US$, $N(\cS_1)\setminus\cS_1$ or $N(\cS_2)\setminus\cS_2$ must contain some dense cell.
\end{enumerate}
As an immediate consequence of this claim, by considering the maximal connected sets of dense cells, we deduce that $\GD$ must have a unique large component, consisting of all cells in $\cD_0$ (note that $\cD_0\ne\emptyset$ by statement~3 in Lemma~\ref{l:P2}). Moreover, $N(\cD_0)$ must touch all four sides of $\US$. Now suppose that $\GB$ has some large component $\cS$. By definition $N(\cS)\setminus\cS$ contains only sparse cells. Then, by the first part of the claim, $N(\cS)$ must touch the four sides of $\US$. Hence, we apply the second part of the claim to $\cS$ and $\cD_0$ to deduce that such large $\cS$ cannot exist.

It just remains to prove the initial claim.
Let $\cS$ be a connected set of cells. Observe that $\bigcup N(\cS)$ is topologically connected (and in particular $N(\cS)$ is a connected set of cells), and that the outer boundary $\gamma$ of $\bigcup N(\cS)$ is a simple closed polygonal path along the grid lines in $\US$ defined by the tessellation. If we remove from $\gamma$ the segments that coincide with some side of $\US$, each connected polygonal path that remains is called a \emph{piece} of $\gamma$. Note that $N(\cS)\setminus\cS$ need not be a connected set of cells. However all cells in $N(\cS)$ along the same piece of $\gamma$ must be contained in the same topological component of $\bigcup(N(\cS)\setminus\cS)$, and thus in the same connected component of $\GC[N(\cS)\setminus\cS]$.

The argument comprises several cases. For each case, a lower bound on the area of some connected component of $\GC[N(\cS)\setminus\cS]$ is given by finding some disjoint subsets of $\US$ of large enough area contained in the union of cells in that component. Then, Lemma~\ref{l:P2} ensures that $N(\cS)\setminus\cS$ contain at least one dense cell.

Given a cell $c$, let $\BNE(c)$ be the set of points at distance at most $(1-4\delta')r$ from the top right corner of $c$ and above and to the right of that corner. The sets $\BNW(c)$, $\BSE(c)$ and $\BSW(c)$ are defined analogously replacing (top, above, right) by  (top, above, left), (bottom, below, right) and (bottom, below, left) respectively. Note that $\BNE(c)$, $\BNW(c)$, $\BSE(c)$ and $\BSW(c)$ are disjoint and contained in $\bigcup(N(c)\setminus\{c\})$.

\setcounter{case}{0}
\begin{case}\lab{c:noside}
Let $\cS\subseteq\cC$ be a connected set of cells which is not small and such that $N(\cS)$ does not touch any side of $\US$.
Since $\cS$ is not small, assume without loss of generality that its vertical extent is greater than $16\delta'r$. Let $c_1$, $c_2$, $c_3$, $c_4$ be respectively the topmost, bottommost, leftmost and rightmost cells in $\cS$ (possibly not all different and not unique). 
Let $A_{\rightarrow}$ be any rectangle of height $16\delta'r$ and width $(1-20\delta')r$ glued to the right of $c_4$ and between the top of $c_1$ and the bottom of $c_2$. Also choose a similar rectangle $A_{\leftarrow}$ of the same dimensions glued to the left of $c_3$, and let $A_{\uparrow}$ and $A_{\downarrow}$ be rectangles of height $(1-4\delta')r$ and width $\delta'r$ placed  on top of, and below, the cells $c_1$ and $c_2$ respectively. By construction, $\BNE(c_1)$, $\BNW(c_1)$, $\BSE(c_2)$, $\BSW(c_2)$, $A_{\uparrow}$, $A_{\downarrow}$, $A_{\leftarrow}$ and $A_{\rightarrow}$ are disjoint and are contained in the same topological component of $\bigcup(N(\cS)\setminus\cS)$ (i.e. the one that touches $\gamma$), which thus has area at least
\[
\pip(1-4\delta')^2r^2  + 2 \delta'r(1-4\delta')r 
+ 32 \delta'r(1-20\delta')r \ge \pip r^2(1+\delta'/3).
\]
Hence, by Lemma~\ref{l:P2}, $N(\cS)\setminus\cS$ must contain some dense cell.
\end{case}
\begin{case}\lab{c:1side}
Let $\cS\subseteq\cC$ be a connected set of cells which is not small and such that $N(\cS)$ touches only one side of $\US$ (assume it is the bottom side). This is very similar to Case~\ref{c:noside}, so we just sketch the main differences in the argument.

If the vertical  extent of $\cS$ is greater than $16\delta'r$, then proceed as in Case~\ref{c:noside} but only consider the sets $\BNE(c_1)$, $\BNW(c_1)$, $A_{\uparrow}$, $A_{\leftarrow}$ and $A_{\rightarrow}$. Otherwise, the horizontal extent of $\cS$ must be greater than $16\delta'r$, and we consider instead the sets $\BNE(c_4)$, $\BNW(c_3)$, $A'_{\uparrow}$, $A'_{\leftarrow}$ and $A'_{\rightarrow}$. Here, $A'_{\leftarrow}$ and $A'_{\rightarrow}$ are rectangles of height $\delta'r$ and width $(1-4\delta')r$ placed to the left and right of cells $c_3$ and $c_4$ respectively, and $A'_{\uparrow}$ is any rectangle of height $(1-20\delta')r$ and width $16\delta'r$ glued on top of $c_1$ and strictly between the left side of $c_3$ and the right side of $c_4$. In both cases, we deduce that the topological component of $\bigcup(N(\cS)\setminus\cS)$ that touches the upper piece of $\gamma$ has area at least $(1+\delta'/6)\pip r^2/2$. Since some cells in this component touch one side of $\US$,   Lemma~\ref{l:P2} implies that $N(\cS)\setminus\cS$ must contain some dense cell.
\end{case}
\begin{case}\lab{c:23sides}
Let $\cS\subseteq\cC$ be a connected set of cells which is not small. Suppose first that $N(\cS)$ touches exactly two sides of $\US$ which are adjacent (say the bottom and the left sides of $\US$). If the horizontal extent of $\cS$ is at most $4r$, then $N(\cS)\setminus\cS$ has some cell inside the $5r\times5r$ square on the bottom left corner of $\US$. But these cells are all dense by Lemma~\ref{l:P2} and we are done. Hence we can assume that $\cS$ has horizontal extent greater than $4r$. In the other cases that $N(\cS)$ touches two non-adjacent sides or three sides of $\US$, we can assume without loss of generality that $N(\cS)$ touches the left and right sides of $\US$ but not the top side. Therefore, in all the cases considered, $\cS$ must contain some cells intersecting each of the five first vertical stripes of width $r$ at the left side of $\US$. Let $c_1$, $c_2$, $c_3$, $c_4$ and $c_5$ be the uppermost cells in $\cS$ intersecting each of the five vertical stripes. These cells are not necessarily all different, but for each $c$ of these, either $\BNW(c)$ or $\BNE(c)$ is completely contained in the corresponding strip.
Thus, the topological component of $\bigcup(N(\cS)\setminus\cS)$ that touches the upper piece of $\gamma$ has area at least $5(1-4\delta')^2\pip r^2/4>(1+1/8)^2\pip r^2$, and by Lemma~\ref{l:P2}, $N(\cS)\setminus\cS$ must contain some dense cell.
\end{case}
\begin{case}\lab{c:4sides}
Let $\cS_1$ and $\cS_2$ be connected sets of cells not adjacent to each other (i.e.\ $\cS_2\cap N(\cS_1)=\emptyset$) and such that both $N(\cS_1)$ and $N(\cS_2)$ touch all four sides of $\US$. Note that by Lemma~\ref{l:P2} all cells inside the $5r\times5r$ square on the top left corner of $\US$ are dense. Assume that none of these cells belongs to $N(\cS_1)\setminus\cS_1$ or $N(\cS_2)\setminus\cS_2$ (otherwise we are done). It could happen that these cells in the top left square are either all in $\cS_1$ or all in $\cS_2$. Assume they are not in $\cS_1$. Then consider, as in Case~\ref{c:23sides}, the uppermost cells $c_1$, $c_2$, $c_3$, $c_4$ and $c_5$ in $\cS_1$ intersecting each of the five first vertical stripes of width $r$ at the left side of $\US$. The same argument shows that the topological component of $\bigcup(N(\cS_1)\setminus\cS_1)$ that touches the upper left piece of $\gamma$ has area at least $(1+1/8)^2\pip r^2$, and Lemma~\ref{l:P2} completes the proof.\qedhere
\end{case}
\end{proof}

Finally, we need to show that bad components  a.a.s.\ have some properties  to be used in the construction of the Hamilton cycles.
Given a component $b$ of $\GB$, let $J=J(b)\subseteq\N$ be the set of indices of points in $\bX$ contained in some cell of $b$. Moreover, for any $r'$, consider the set $J'=J'(b,r')=N_{\GXrp}(J)\setminus J$ (i.e.\ the set of strict neighbours of $J$ in a random geometric graph of radius $r'$).
\begin{lem}\lab{l:ring}
For a small enough constant $\delta>0$, any constant $\lambda\in\real$, $r$ defined by $\pip nr^2=\log n+(k-1)\log\log n + \lambda$ and $r\le r'\le(1+1/32)r$, the following is a.a.s.\ true.
For each small component $b$ of $\GB$, there exists a connected set of dense cells $\cR(b)\subseteq\cD_0$ of size $0<|\cR(b)|\le10/\delta^2$ such that
\begin{enumerate}
\item for every $i\in J'(b,r')$, the cell containing $X_i$ is adjacent to some cell in $\cR(b)$, and
\item $\cR(b)\cap\cR(\widetilde b)=\emptyset$ and $J'(b,r')\cap J'(\widetilde b,r')=\emptyset$, for any other small component $\widetilde b$ of $\GB$.
different from $b$.
\end{enumerate}
\end{lem}
\begin{proof}
Let $b$ be a small component of $\GB$, and let $g$ be any $16\times16$ grid covering $b$. Let $O$ denote the geometric centre of the grid $g$, and let $\cS$ be the set of cells which have some point at distance between $3r/4$ and $3r/2$ from $O$. Take as $\cR(b)$ the subset $\cR=\cS\cap\cD$ formed by the dense cells in $\cS$. This set will be shown to have all the desired properties. (Note that the size of $\cR$ is $|\cR|\le|\cS|<10/\delta^2$.)

Consider a coarser tessellation of $\US$ into larger squares of side $\lfloor1/(16\delta')\rfloor\delta'r$ (each square containing exactly $\lfloor1/(16\delta')\rfloor^2$ cells). We refer to each square both as a subset of $\US$, and as the set of cells it contains.
Let $\cQ$ be the set of squares of the coarser tessellation that contain at least one point at distance exactly $5r/4$ from $O$.
By construction, all squares in $\cQ$ are contained inside $\cS$. Moreover, we claim that all squares in $\cQ$ contain some dense cell. In fact, by choosing $\delta$ sufficiently small, we can guarantee that each square $q\in\cQ$ has no intersection with $N(b)\setminus b$, and thus $q\cup(N(b)\setminus b)$ is a connected set of cells of area at least
\[
\pip(1-34\delta')^2r^2 + \lfloor1/(16\delta')\rfloor^2\delta'^2r^2 \ge (\pip+1/257)r^2.
\]
Hence, assuming that statement 1 in Lemma~\ref{l:P2} holds, $q\cup(N(b)\setminus b)$ must contain some dense cell, which must be in $q$ since $N(b)\setminus b$ does not contain any.

Since the union of squares in $\cQ$ is topologically connected, and each pair of cells lying in topologically adjacent squares of $\cQ$ are also adjacent in $\GC$, the dense cells in squares of $\cQ$ induce a connected set of cells. Moreover, for any other cell $c$ in $\cS$ there is some square $q\in\cQ$ such that $c$ is adjacent to all cells in $q$. Hence,
$N(\cR)\supseteq\cS$, and also $\cR$ induces a connected set of cells. Since $\cR$ cannot be embedded in a $16\times16$ grid of cells, $\cR$ must be contained in $\cD_0$.

Now consider any vertex $i\in J'=J'(b,r')$.
If $\dist(X_i,O)\le3r/8$, then the cell $c$ containing $X_i$ must be in $N(b)\setminus b$. Therefore, since $b$ is a component of $\GB$, $c$ must be sparse but adjacent to some dense cell $d\in\cD_0$. By construction, any point in $d$ must be at distance between $(1-34\delta')r$ and $(11/8+2\delta')r$ from $O$, so $d\in\cR$. Otherwise, suppose that $\dist(X_i,O)>3r/8$. We also have $\dist(X_i,O)\le(1+1/32+16\delta')r$, since $i\in J'$. Then the cell $c$ containing $X_i$ must be adjacent to all cells in some square $q\in\cQ$, and in particular to some dense cell in $\cR$.

To verify the other requirements, define $\cQ'$ to be the set of squares of the coarser tessellation with some point at distance exactly $7r/4$ from $O$. The same argument we used for $\cQ$ shows that all squares in $\cQ'$ contain some dense cell. Let $\cR'$ be the set of dense cells in squares of $\cQ'$.
Then it is immediate to verify that any point in a cell $c$ of some other small component $\widetilde b\ne b$ of $\GB$ must be at distance at least $41r/16$ from $O$ since otherwise $c$ would be adjacent to some cell in $b$, $\cR$ or $\cR'$. All remaining statements follow easily from that.
\end{proof}

\section{Packing linear forests in bipartite graphs}
\lab{s:colouring}

 A {\em factorisation} of a graph is the set of subgraphs induced by a partition of the edge set. A {\em hamiltonian decomposition} of a graph is a factorisation in which at most one subgraph is a perfect matching, and all the remaining ones  are   Hamilton cycles. We call a matching that contains an edge of each of the Hamilton cycles in the decomposition a {\em transversal} of the decomposition. (Note that the transversal does not contain an edge of the perfect matching.) The construction of a hamiltonian decomposition in the following result  is well known.  We will use features of the construction in the proof of Lemma~\ref{l:colouring}, and we use the transversal in Section~\ref{s:build}.
\begin{lem}\lab{l:decomp}
Every complete graph has a hamiltonian decomposition with a transversal.
\end{lem}
\noindent Note that the number of Hamilton cycles in such a  decomposition of $K_{k+1}$ will be $\lfloor k/2\rfloor$, and thus for $k$ odd the transversal is not quite a perfect matching.
\begin{proof}
First, for $k$ even, consider the complete graph $K_{k+1}$ on the vertices $\{1,2,\ldots,k,*\}$. We shall first colour the edges of $K_{k+1}$.  Expressions referring to vertex labels other than $*$ are interpreted mod $k$ and expressions referring to colour labels are mod $ k/2$. (In this paper, mod denotes taking the remainder on division.)

For each pair of vertices $u$ and $v$ in $\{1,2,\ldots,k\}$, assign the colour 
\begin{equation}\lab{colour}
\lceil (u+v)/2 \rceil
\end{equation}
 (mod $k/2$ of course) to the edge  $u v$. Also, assign colour $i$ to the edges from $*$ to both vertices $i$ and $i+k/2$. See Figure~\ref{f:coloured}.

\begin{figure}[htb]
\begin{centering}
\pic
\lab{f:coloured}
\end{centering}
\end{figure}
 
It is easy to check that, for each $i\in\{1,\ldots,k/2\}$, the edges receiving colour $i$ form a $(k+1)$-cycle $(v_0,\ldots,v_k)$ where   
\[
v_0=*,\qquad v_1=i,\qquad v_{t+1}=v_t+(-1)^tt, \quad\forall t\in\{1,\ldots,k-1\},
\]
or equivalently
\[
v_0=*,\qquad   v_t=i-(-1)^t\lfloor t/2\rfloor, \quad\forall t\in\{1,\ldots,k\}.
\]
Thus, the colouring induces a factorisation of $K_{k+1}$ into $k/2$ Hamilton cycles of colours $1,\ldots,k/2$, giving the required hamiltonian decomposition.

When $k$ is congruent to 2 mod 4,  the set of edges $\{2i,2i+1\}$ ($i=0,\ldots, k/2-1$) is a transversal.    When $k$ is   divisible by 4, one transversal uses the edges $\{2i,2i+1\}$ ($i=0,\ldots, k/4-1$), the edge from $*$ to $k/2$, and the edges $\{k/2 + 2i-1,k/2 + 2i\}$ ($i=1,\ldots, k/4-1$).   

For odd $k$,  a perfect matching needs to be included. There is a similar colouring scheme, using the colours $1,\ldots , (k+1)/2$, where colours are taken mod $ (k+1)/2$.  In this case, colour $\lceil (u+v \mod k)/2 \rceil$
is on the  edge $uv$ (note we assume by convention that $u+v \mod k\in\{0,\ldots, k-1\}$),
each  colour $i$ ($i\in\{1,\ldots , (k-1)/2\}$) is on the edge from $i$ to $*$ , and each  colour $(k+1)/2-i$ ($i\in\{0,\ldots , (k-1)/2\}$) is on the edge from $k-i$ to $*$.
The edges of colour $(k+1)/2$ form a perfect matching, and each of the other colours gives a Hamilton cycle.  
Finally, the transversal for odd $k$ is easily found, similar to the even $k$ case, using more or less every second edge of the form $\{i,i+1\}$. 
\end{proof}

\noindent
{\bf Proof of Lemma~\ref{l:colouring}}
Since the statement is trivial for $j=1$, we may assume that $j\ge2$. Firstly, we deal with the case of even $k$. 
We will  assign colours in $\{1,\ldots,k/2\}$ to a subset of the edges of $G$ such that the colour classes determine the decomposition into linear forests. With a slight abuse of standard notation we will call this an edge colouring of $G$. For simplicity of exposition, we assume that the degree lower bounds for $H$ are all met precisely. It will be evident from the proof that if any lower bounds on the degrees of vertices in $J$, that are specified in the lemma statement,  are exceeded, it can only help by giving more choices in various steps. Hence, we may assume that the apex vertex has degree exactly $\ell+1$, and all others have degree exactly $\ell$.
Label the apex vertex in $J$ with `$*$', and label the other vertices $1,\ldots, j-1$. 

First colour the edges of $K_{k+1}$ as in the proof of Lemma~\ref{l:decomp}, using the same labels on vertices as in that proof.
If we delete the vertices in the set $L=\{j,j+1,\ldots,k\}$ of $K_{k+1}$, the remaining vertices of $K_{k+1}$ are in $J$ and the edges between them in $G$ can inherit the colours from $K_{k+1}$. The set $L$ should be regarded as a set of vertices sitting outside $G$. For $v\in J$, edges to other vertices in $J$ are coloured, but the colour on any edge to $L$ is  `missing' from that vertex in $G$.  When we speak of missing colours, we count them with multiplicities: if a $v$ has no edge of    colour $i$ in $G$ (i.e.\ two edges of   colour $i$ join $v$ to $L$) then   colour $i$ is missing at $v$ with multiplicity 2.  Of course, each vertex is incident with two edges of each colour in $K_{k+1}$. The missing colours need to be assigned to edges of $H$, which go between $J$ and $B$, since each vertex of $J$ must have degree 2 in each of the final linear forests.

We will use a greedy colouring procedure to assign the missing colours to the edges of   $H$, and thereby complete the  desired edge colouring of $G$. The requirement is simply that each colour class must induce a linear forest. (By taking care of the missing colours we are ensuring that all vertices of $J$ have two incident edges of each colour.)
The procedure treats the vertices in the order $1,2,\ldots, j-1$, and finally  the apex vertex,  $*$.  Note that, so far, only edges within $J$ are coloured, and each colour induces a set of paths. This is because the edges of any given colour induce a proper subgraph of the original Hamilton cycle of that colour. The procedure makes the assignments of new edges one by one, so it simply has to avoid all monochromatic cycles at each step, and terminate with each   vertex in $J$ having precisely two incident edges of each colour in $G$, just like they do in $K_{k+1}$, and each vertex in $B$ ($=V(H)\setminus J$) having at most two incident edges of any given colour.

The colouring procedure is defined inductively and requires several observations along the way.
We will first fix a vertex $v\in J\setminus\{*\}$ (i.e. $1\le v\le j-1$)  and specify how the procedure treats the edges of $H$ incident with   $v$. We may  assume inductively that there are no missing colours at vertices $1,\ldots, v-1$, i.e.\ all these vertices are incident with two edges of each colour, and furthermore that at each step  the set of edges with any given colour induces a subgraph consisting of disjoint paths.  The colours missing at $v$ are those on the edges from $v$
to $L$. Let  $i$ be a colour that is missing at $v$. There are two cases. In the first case, $i$ is on just one edge at $v$, from $v$ to $k-t$ say, and either $t=0$ and no edges of  colour $i$ are already present in $H$, or $t=k-j=\ell-1$. In the second case, $i$ is on two edges, from $v$ to $k-t-1$ and $k-t$, for some $0\le t
\le \ell-2$.  (In both cases of course  $t$ can be computed from the colour formula~(\ref{colour}).) Let us refer to this colour $i$ as $i_t$.  

The specification of how the colouring procedure treats $v$ is as follows.  The missing colours $i_t$ are treated one by one in decreasing order of $t$. At any point, let $G_i$ denote the subgraph of $H$ induced by the edges   coloured $i$. Each missing colour $i$ of multiplicity $\delta$, in its turn, is assigned to $\delta$ uncoloured edges of $H$ incident with $v$,  in any manner such that: 
\begin{description}
\item{(i)} $G_i$ remains a linear forest, 
\item{(ii)}  $G_i$ is not a connected graph unless all choices satisfying (i) cause $G_i$ to be connected. 
\end{description} 
Actually, we only need to invoke   rule (ii) when $t = \ell-3$, but it does no harm to enforce it in each step.

For the apex  vertex $*$, which is treated last, the rule is simpler. The procedure assigns the missing colours   to uncoloured edges of $H$ incident with $*$, greedily  subject to rule (i), and the order of treatment of the colours is determined at the start as follows: colours that  already appear on more edges are treated earlier.

To verify that the colouring procedure must terminate with each colour inducing a linear forest, we argue inductively for $v\in J$, $v\ne *$, in the order of treatment.  The apex vertex is considered last. We show that the linear forest condition holds after each vertex is treated.  The argument for the inductive step also applies to the initial step, where $v=1$. 

So, for a vertex $v\in J\setminus \{*\}$, consider the point at which the procedure is treating the colour $i_t$ defined above. Any edges  coloured $i_t$ at earlier steps of the inductive procedure must have been edges from $\{v-t-1,v-t,\ldots, v-1\}\setminus L$ to $L$, and these must go to the vertices  $k-t+1, \ldots, k$. See Figure~\ref{f:coloured}, where the colour $i_t$ is shown in purple. Each vertex in $L$ is incident with at most two such edges of  colour $i_t$, as is each vertex in $\{v-t-1,v-t,\ldots, v-1\}\setminus L$. In particular, $v-t-1$ is  incident with exactly one such edge, the one  joining it to vertex $k$.  Hence, the number of times that colour $i_t$ is missing (counted with multiplicities) on the vertices already treated   is at most $2t+1$. Consequently, $2t+1$ is an upper bound on the number of edges of colour $i_t$ in $H$ at the   point in the  procedure where vertex $v$ is about to be treated. It is not necessary, but may help to note that the distinct colours missing at $v$ are either $i_0, i_2,\ldots, i_{m}$ (where $m=2\lfloor (\ell-1)/2 \rfloor$), if $i_0$ has multiplicity 2, or $i_0, i_1, i_3,\ldots, i_{m}$ (where $m=2\lfloor \ell/2 \rfloor-1$) if $i_0$ has multiplicity 1. 

Recall that $v$ has degree $\ell$ in $H$. Let $\delta$ denote the multiplicity of $i_t$ as a missing colour at $v$.
 When the procedure is about to treat colour $i_t$ at $v$, the colours on edges from $v$ to $j, j+1,\ldots, k-t-\delta$ have already been assigned. This means that $k-t-\delta-j+1$ of the edges   incident with $v$ are already assigned colours. As $d_H(v)=\ell=k-j+1$, there are precisely $t+\delta$  edges of $H$ incident with $v$ that remain uncoloured. Recall that $H$ has at most $2t+1$ edges already coloured $i_t$. Since $H$ is bipartite, at most $t$ vertices in $B$ can have degree 2 in $G_{i_t}$. Hence, there must be at least $\delta$ (which is either 1 or 2, precisely the multiplicity of $i_t$) uncoloured edges from $v$ to vertices of $B$ having degree at most 1 in $G_{i_t}$.  So it is possible to assign colour $i_t$ to  $\delta$ of these edges at this point without creating any vertices of degree greater than 2 in $G_{i_t}$. We need to show that this can always be done so as to satisfy condition (i), that is, without creating a cycle in $G_{i_t}$.  

Before proceeding, we need to understand the ways that such a cycle can form. If $\delta=2$, then two uncoloured edges joining $v$ to  $N_H(v)$ must be picked with the purpose of colouring them  $i_t$, and a cycle is created if and only if the two end-vertices in $N_H(v)$ are the two ends of a path in $G_{i_t}$. If $\delta=1$, a cycle is created if and only if the edge picked is the end-vertex of a path in $G_{i_t}$, of which $v$ is the other end-vertex.  

Let $U$ denote the set of vertices in $N_H(v)$ that are joined by uncoloured edges to $v$. It is always possible to avoid the cycle in question if  $U$ contains  more than $\delta$ vertices of degree less than 2 in $G_{i_t}$. For, if this is true in the $\delta=2$ case, then, even if one end of a path of $G_{i_t}$ is picked for the first edge to be coloured $i_t$, the second edge can avoid the other end of the same path. The case $\delta=1$   is even easier. There is similarly no problem if $U$ contains a vertex of degree 0 in $G_{i_t}$. 

So, if the the edges of colour $i_t$ do not form a linear forest  after the procedure treats $v$, we may assume that $U$ contains precisely $\delta$ vertices of degree 1 in $G_{i_t}$, and that the other $t$ vertices of $U$ (recall that $|U|=t+\delta$) have degree 2 in $G_{i_t}$. In particular, the number of edges of colour $i_t$ in $H$ is precisely $2t+\delta$. However, above we deduced that it is at most $2t+1$. Hence, $\delta=1$, and there must be precisely $t$ vertices in $B$ of degree 2 in $G_{i_t}$. That is, $i_t$ is the colour of only one edge from $v$ to $L$. As explained above, this happens only if, on the one hand, $t=0$ and there are no edges already of colour $i_t$ in $H$, or, on the other hand, $t=\ell-1$. The first case contradicts the fact that $U$  contains precisely $\delta$ vertices of degree 1 in $G_{i_t}$. So the second case holds, and $t$ must be equal to $\ell-1$. This means that $u_t$ is the first colour being treated by the process at $v$, and furthermore, the $t+1$ vertices in $U$ comprise exactly $N_H(v)$. 

It was observed above that the only other vertices sending edges to $L$  of colour $i_t$ are  $\{v-t-1,v-t,\ldots, v-1\}\setminus L$ and the maximum number of such edges is $2t+1$. The edges in $H$ of colour $i_t$ come from these vertices. Since this upper bound is achieved, the situation is tight: $v-t-1$ is incident with one such edge in $H$, and the $t$ vertices in $\{v-t-1,v-t,\ldots, v-1\}$ are incident with two each. Recalling that $U$ has precisely one vertex of degree 1 in  $G_{i_t}$ and the rest have degree 2, we conclude that the graph $G_{i_t}\cap H$ has precisely two vertices of degree 1 and the rest of degree 2. Since by induction it contains no cycle, it is a path $P$, and one end-vertex of $P$ is $v-t-1$. Considering the original colouring of $J$, the only other edges of colour $i_t$ in $G$ at this point form a path on the vertices $\{1,\ldots, v-t-1\}$ that is vertex-disjoint from $P$ except for the vertex $ v-t-1$. Hence, $G_{i_t}$ is a path and is hence connected.

Shortly we will also need a different observation. Note that if $t=0$, which means $\ell=1$, then the hypotheses imply that $v$ has precisely one neighbour in $B$, which is distinct from the neighbours of $1,2,\ldots, v-1$. So this cannot be the case. It follows that $t+1=\ell\ge 2$.

Since there is at least one edge already coloured $i_t$, we have $v\ge 2$. Hence, the vertex $v-1$ was already treated, and has two edges of this same  colour $i_t$ in $H$ joining it to $B$. Call them $x$ and $y$. Shift the focus back to the time that the procedure, when treating $v-1$, coloured $x$ and $y$. After they were coloured, as observed above, $G_{i_t}$ became   connected. By rule (ii), it was necessary that that no other choice of two uncoloured edges from 
$v-1$ to $B$ could avoid creating a connected graph $G_{i_t}$. However, $x$ and $y$ are two consecutive edges in the path $P$ defined above. Let $w$ denote any vertex of $N(v-1)\cap B\setminus N(v)$, which must exist by the second hypothesis of the lemma. Note that $w$ is adjacent to no edges of colour $i_t$. Also note that  $i_t$ is the first colour treated by the procedure when dealing with the vertex $v_{t-1}$, since this is the colour of the edge from $v_{t-1}$ to the vertex $j = k-t$. Hence, at that point in the procedure, all edges of $H$ incident with $v_{t-1}$ are uncoloured, in particular the edge to $w$.

We consider two cases. Firstly, if $x$ or $y$ is incident with an end-vertex $u$ of $P$, it must be that $u$ lies in $B$. There are no other edges of colour $i_t$ incident with $u$. So, instead of colouring $x$ and $y$ using the colour $i_t$, the edges to $u$ and $w$ could have been used instead, and these would form a separate path in $G_{i_t}$, making it disconnected, which is a contradiction by rule (ii). If, on the other hand, $x$ and $y$ are elsewhere in $P$, then   the graph $P'$  induced by the edges of $P$ other than $x$ and $y$ is disconnected. In this case, the procedure could have placed the colour $i_t$ on the edge $x$, and on the edge from  $v_{t-1}$ to $w$, again a contradiction since $G_{i_t}$ becomes disconnected.

It follows that our assumptions about $U$ above (that it contains precisely $\delta$ vertices of degree 1, etc.)\ are false. This concludes what was required to show that after the procedure treats $v\in J\setminus \{*\}$, the edges of each colour induce a linear forest. 

Finally, we turn to the apex vertex, $*$.  Note that the multiset of colours missing at $*$ is precisely $\{j,j+1,\ldots,k\}$:  each such colour $i$ lies on an   edge from $*$ to the vertex $i$. 

We first consider $j>k/2$. In this case the missing colours at $*$ all have multiplicity 1, so it is a simpler situation than for smaller $j$. If we list the colours in the following order: $i_1=j$, $i_2=k$, $i_3=j+1$, $i_4=k-1$, $i_5=j+2$ and so on, it is easily seen that  the number of edges already coloured $i_t$ is precisely $2\ell+1-2t$ ($1\le t\le \ell$). So the colouring procedure treats them in the order $i_1,i_2,\ldots,i_\ell$. Since $*$ has degree $\ell+1$ in $H$, the number of uncoloured incident edges it has when treating colour $i_t$ is $\ell+2-t$. At this point, there are only enough edges already coloured $i_t$ to create at most $\ell-t$ vertices of degree 2 in $B$. Hence, there are still at least another two vertices in $B$ joined to $*$ by uncoloured edges.The only way to create a cycle in colouring one of these edges $i_t$ is to use the edge that joins to the other end of the unique path in $G_{i_t}$ that presently begins with $*$. Hence, there is yet another edge available to safely colour $i_t$ so as to maintain the linear forest condition.

It only remains to show that a similar statement holds for the case $j\le k/2$. Here, the colours $1,\ldots ,j-1$ are missing with multiplicity 1 at $*$, and the  colours $j,\ldots ,k/2$ are missing with multiplicity 2. At this point in the colouring procedure, the numbers of edges of colours $1,j-1,2,j-3, \ldots$ are $2j-3, 2j-5, 2j-7,\ldots$ respectively. So the procedure will treat these colours in that order. On the other hand, for each $j\le i\le k/2$, the number of edges of colour $i$ is precisely $2j-2$, and these were placed in pairs, two at a time from each of the vertices $1,\ldots, j-1$. So the procedure treats these colours first. When treating  colour  $i\in\{j,\ldots ,k/2\}$, the number of uncoloured edges from $*$ to $B$ will still remain at least $\ell + 1 -(k-2j)=j+2$ before each colouring step. With only $2j-2$ edges already of   colour $i$ in $H$, it is easy to assign two more edges of $H$ incident with $*$ the colour $i$ without creating a cycle or vertex of degree bigger than 2 in $G_i$: either $G_i$ is a single path with $j$ vertices in $B$, in which case there are two edges to vertices of degree 0 in $G_i$, or $G_i$ has at least two components, in which case the procedure can avoid the at most $j-1$ vertices of degree 2 in $G_i$ and join instead to two of degree 1 or 0, and not in the same component of $G_i$. For the remaining colours in $\{1,\ldots, j-1\}$, set $i_1=1$, $i_2=j$, $i_3=2$ etc in the order given above, so that the number of edges of colour $i_t$ is $2j+1-2t$. So, the argument as in the case $j>k/2$ (with   $\ell$ replaced by $j$ in the appropriate places) shows that the colours can be assigned as required.

Finally, we consider the case of odd $k$, which will be reduced to the even case. Note that the colouring in Figure~\ref{f:coloured} is rotationally symmetric, and thus in the even $k$ case, we can apply the same colouring procedure if the labels of the vertices of the auxiliary $K_{k+1}$ are shifted by some quantity (i.e.\ we can use the labels in $\{a+1,\ldots,a+j-1,*\}$ for the vertices in $J$ and delete from $K_{k+1}$ the vertices in $L=\{a+j,a+j+1,\ldots,a+k\}$). So, for the case of odd $k$, label the apex vertex in $J$ with $*$, and label the other vertices $\lceil \ell/2\rceil,\ldots,\lceil \ell/2\rceil+j-2$. Then, colour the edges of an auxiliary $K_{k+1}$ as in the proof of Lemma~\ref{l:decomp}, using the same labels on the vertices as in the proof. We refer to the colour $(k+1)/2$ as the {\em match} color. Delete from $K_{k+1}$ the vertices in the set $L=\{\lceil\ell/2\rceil-i\mod k \st 1\le i\le \ell\}$, and identify the set of remaining vertices with $J$. As in the case of even $k$, we must assign all the missing colours to the edges of $H$ with the same requirements as before on the non-match colours, but also the match colour must induce a matching in $H$. Note that  this last condition is trivially satisfied since, by our specific choice of $J$ and $L$, at most one edge coloured with the matching colour crosses between $J$ and $L$.

\begin{figure}[htb]
\begin{centering}
\includegraphics[height=10cm]{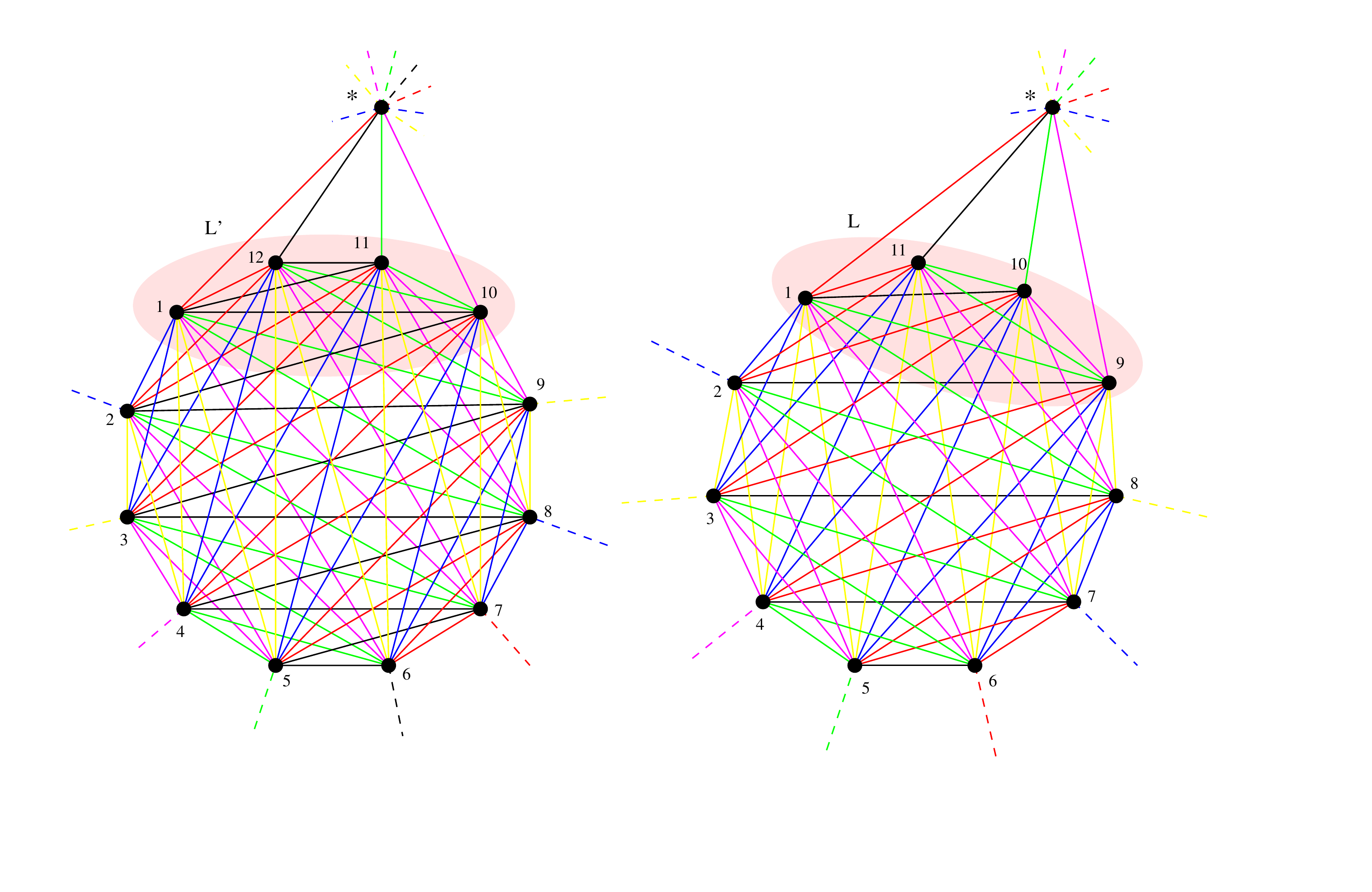}
\caption{\em Reduction of the case $k=11$ to the case $k+1=12$.}
\lab{f:evenodd}
\end{centering}
\end{figure}

To find the required colour assignment to the edges of $H$, consider the bipartite graph $H'$ with parts $J'$ and $B$ resulting from adding a new vertex labelled $\lceil\ell/2\rceil+j-1$ to $J$, with $\ell$ edges from the new vertex to arbitrary vertices of $B$. We shall use the previous greedy procedure to colour the edges of $H'$ noting that $|J'|+\ell-1=k+1$ is even. Take a copy of $K_{k+2}$ (disjoint from $K_{k+1}$) with labels $1,\ldots,k+1,*$ on the vertices, and colour the edges as in the proof of Lemma~\ref{l:decomp}. Partition the vertex set of $K_{k+2}$ into $J' = \{ \lceil\ell/2\rceil,\ldots,\lceil\ell/2\rceil+j-1,* \}$ and $L'=\{\lceil \ell/2\rceil-i \mod k+1\st 1\le i\le \ell\}$, and delete the vertices in $L'$. Colour the edges of $H'$ using the previously described colouring procedure (with all vertex labels shifted by the appropriate constant, so that the vertices in $L'$ have the correct labels). Even though the edge colours used in $K_{k+1}$ and $K_{k+2}$ are different, it is easy to check that, for each $i\in\{\lceil \ell/2\rceil,\ldots,\lceil \ell/2\rceil+j-2,*\}$, the vertex  of $J$ that is labelled $i$ misses exactly the same set of colours as the vertex  of $J'$ labelled $i$ (see Figure~\ref{f:evenodd} for a visual illustration).
Therefore we may simply obtain the edge colouring of $H'$ from the algorithm used for even $k$ applied to $H'$, and then restrict this colouring to   $H$ to obtain a colouring that satisfies the desired properties.
 \qed


\section{Building Hamilton cycles and a perfect matching}
\lab{s:build}
In this section, we use the results in the earlier lemmas to prove Theorem~\ref{t:main}. We first give a complete proof for $k$ even and then provide the extra pieces of argument required for $k$ odd.
\medskip

\noindent
(i) {\bf Proof for $k$  even.}
\smallskip

Let $\epsilon>0$ be arbitrarily small. Let us choose a large enough constant $\lambda_0>0$ such that $e^{-e^{\lambda_0}} < \epsilon$ and $e^{-e^{-\lambda_0}} > 1-\epsilon$. Set 
$$
r_l=\sqrt{\frac{\log n + (k-1)\log\log n - \lambda_0}{\pip n}},
$$ 
$$
r_u=\sqrt{\frac{\log n + (k-1)\log\log n + \lambda_0}{\pip n}}.
$$
{From}~\cite{Pe99}, we know that $\pr(\GXrl \text{ $k$-connected}) \sim e^{-e^{\lambda_0}}<\epsilon$ and $\pr(\GXru \text{ $k$-connected}) \sim e^{-e^{-\lambda_0}}>1-\epsilon$, so looking at the evolution of $\GXr$ for $0\le r\le \Vert(1,1)\Vert_p$, the probability that it becomes $k$ connected somewhere between $r_l$ and $r_u$ is greater than $>1-2\epsilon$. Call this point $r_k$. By the results in Section~\ref{s:assprop}, we may assume that the properties described in Lemmas~\ref{l:nolarge} and~\ref{l:ring} hold for $\lambda=-\lambda_0$ (thus $r=r_l$) and some $\delta$, and also that the property in Lemma~\ref{l:extradeg} holds for $r=r_k$ and $\eta=32\delta$. 
So we may assume $\bX$ to be an arbitrary fixed set of $n$ points in $\US$ in general position and satisfying these properties. The proof is completed by giving a deterministic construction of $k/2$ edge-disjoint Hamilton cycles for the geometric graph $\GXrk$. Most edges will be of length at most $r_l$ but we shall use a few of length between $r_l$ and $r_k$. (The last edges creating  $k$-connectivity arrive during this period, and they are of course necessary to construct $k/2$ edge-disjoint Hamilton cycles.)
We define the edges of each Hamilton cycle by colouring some of the edges of $\GXrk$, using colours $1,\ldots , k/2$, such that each of these colour classes induces a Hamilton cycle.

We take $r=r_l$ (except at special points in the argument) and define $\GC$, $\cD$, $\cB$ and so on accordingly (see Section~\ref{s:assprop}).
Let $T$ be a spanning tree of the largest component $\cD_0$ of $\GD$.  Next, double each edge of $T$ to get an Eulerian multigraph $F$. The   vertex degrees in   $T$ are bounded above by $\Delta$, so those  in $F$ are bounded above by $2\Delta$.  Next, pick an Eulerian circuit $C$ of $F$.

Henceforth, we have no need to consider points in $\US$  that are not members of $\bX$. So, points in $\bX$ contained in some cell $c$ will simply be referred to as points in $c$, and they will be often identified with their corresponding vertices in $\GXrl$ or $\GXrk$. Also, the term \emph{dense cell} will refer only to cells in $\cD_0$, thus excluding these dense cells contained in bad components. For descriptive purposes, we split the rest of the argument into two parts, first treating the case that there are no bad cells, i.e.\ $\cB$ is empty. For this we only need the edges of $\GXrl$. Then we will show how the construction is easily modified to handle the bad components, using some  edges of $\GXrk$.
\medskip

\noindent
{\bf Part 1. $\cB$ is empty.}
\smallskip

In this case, the rest of the proof  involves two steps, which will be used in different forms during the later arguments.
\medskip

\noindent
{\em Step 1. Turning the circuits into cycles}
\smallskip

 The subgraph of $\GXrl$ induced by the points contained in any dense cell is complete and has many more than $k$ vertices.  Lemma~\ref{l:decomp} provides $k/2$ edge-disjoint Hamilton cycles in this subgraph. In fact, it provides more; we just choose a subset of the Hamilton cycles that are given by that lemma. The separate cycles in all the dense cells will be `broken' and  rejoined together using $C$ as a template. In the following discussion we assume $C$ is oriented, so we may speak of incoming and outgoing edges of $C$ with respect to a cell. 

For any  dense cell $c$, the deletion of $c$ from $C$ breaks $C$ up into a number of paths $P_i$. For colour $1$, do the following. Associate each path $P_i$ with an edge $z_i$ that joins two points in $c$ and has already been   coloured $1$, using a different edge $z_i$ for each path $P_i$.  Uncolour the  edges $z_i$, and associate the outgoing and incoming edges of the path $P_i$ (with respect to the cell $c$) each with an endpoint of $z_i$. After doing the same for all dense cells, every edge $cd$ of $C$, where $c$ and $d$ are cells, has now been associated with two points, one in $c$ and one in $d$. Colour the edge joining these two points using colour 1. Doing this for all edges of $C$ clearly joins up all the edges coloured 1 into one big cycle using all points in the dense cells.

Now do the same with colours $2,\ldots, k/2$, one after another, but each time being careful to use edges $z_i$ in each cell that are not adjacent to such edges used with any of the previous colours.   This is easily done because using an edge for one colour eliminates at most four edges of another colour (as the graph induced by edges of a given colour has maximum degree at most 2). So the process can be carried out if $M$ is greater than $2k\Delta$.

\medskip

\noindent
 {\em Step 2. Extending the cycles into the sparse cells}
\smallskip

There are now $k/2$ edge-disjoint coloured cycles, one of each colour, and each cycle uses precisely all the points in dense cells.  Note that within each dense cell, there are still an arbitrarily large number (depending on $M$) of {\em spare} edges of each colour, left over from the original application of Lemma~\ref{l:decomp}.
To prepare for extending  the cycles into the sparse cells, we will break the cycles at these spare edges.

Let $c$ be any sparse cell. By the definition of $\cB$ and our assumption that $\cB$ has no cells, there is a dense cell, say $c'$,  adjacent to $c$ in $\GC$. 
 If $c$ contains at most $2k$ points,  for each vertex $v$ of the geometric graph inside $c$ do the following. Choose a spare edge $z$ inside $c'$ of colour 1, uncolour the spare edge $z$, and colour the two edges from the endpoints of $z$ to $v$ with the colour 1. Any edges of different colours adjacent to $z$ should be deemed not spare after use. Then repeat for each of the other colours. After this, the edges of any given colour form a cycle containing all points in dense cells and in $c$.

On the other hand, if $c$ contains more than $2k$ points, the above process could potentially require too many spare edges, so we must do something else.  By Lemma~\ref{l:decomp}, we can specify $k/2$ edge-disjoint Hamilton cycles around the points in $c$, one of each of the colours. One can then greedily choose an independent set of edges, one of each colour. (This is easily seen by noting that  choosing an edge knocks out at most four adjacent edges with any particular colour. Alternatively, by a more careful argument which we give later, it can be shown that the same holds  as long as $c$ contains more than $k$ points.) These edges can be  matched up with $k/2$ spare edges that have both endpoints in $c'$, and then each of the coloured cycles is easily extended by uncolouring each matched pair of edges and appropriately colouring the edges joining their endpoints. Again for this case, the edges of any given colour form a cycle containing all the points in dense cells and in $c$.
 
This process can be repeated for each sparse cell. Since each dense cell has at most $\Delta$ neighbours in $\GC$, the total number of spare edges required of any one colour in any dense cell can be crudely bounded above  by $2k\Delta$, which is the same as the upper bound on the number of points already used up. Thus, for $M$ sufficiently large ($4k\Delta$ should do), there will be a sufficient number of spare edges to finish with a cycle of each colour through all points in $\GXrl$, using only the edges of $\GXrl$.
This finishes Step 2 and the proof in the case that  $\cB$ is empty.

We now turn our attention to the (much more difficult) case that $\cB$ is nonempty, for which we use an appropriate modification of the above argument. 

\medskip

\noindent
{\bf Part 2. The general case: $\cB$ can be nonempty.}
\smallskip

For this, we will need to use some edges of $\GXrk$ that are not present in $\GXrl$, but the definition of all structures (such as bad components) remains as determined by the graph $\GXrl$. Recall the Eulerian circuit $C$ chosen at the start of the proof. This circuit gives a directed tour of all dense cells in the graph $\GC$. We will first extend it to a circuit that includes routes through each bad component, and later perform modified versions of Steps 1 and 2 described above. 

Pick one such bad component $b$, which must be small by Lemma~\ref{l:nolarge}, and let $\cR=\cR(b)$  be a set of cells as in Lemma~\ref{l:ring}. Recall that $0<|\cR|<10/\delta^2$. To take care of $b$, we will work entirely in $\cR$ and the bad component $b$. Let $J$ denote the set of points in cells in $b$ and set $j=|J|$ (assume that $j>0$, since otherwise $b$ has no role in our argument). The subgraph of $\GXrl$ induced by $J$ is a copy of $K_j$, since $b$ is small and can be embedded in a $16\times16$ grid of cells (and assuming that $32\delta<1$).  Now consider the graph  $\GXrk$, which by definition is $k$-connected, and let $J'= N_{\GXrk}(J)\setminus J$. Let $H$ denote the induced bipartite subgraph of $\GXrk$ with parts $J$ and $J'$, and let $G\subseteq\GXrk$ be the union of $H$ with the clique on vertex set $J$. Note for later reference that the set $J'$ can possibly contain vertices in dense cells: although no cell in $b$ is adjacent to a dense cell, points in it can be adjacent to points in a dense cell.
\medskip

\noindent 
{\em Claim 1.\ }
$G$ contains $k/2$ pairwise edge-disjoint linear forests $F_1,\ldots, F_{k/2}$, such that in each forest

\noindent
(a) all vertices in $J$ have degree $2$, and

\noindent
(b) at most $2k$ vertices in $J'$ are contained in any path.

\smallskip

 To prove the claim,
we consider two cases, the second being much more difficult since it requires Lemma~\ref{l:colouring}. 
\medskip

\noindent
{\em Case 1:   $j>k$.}
\smallskip

Since $\GXrk$ is $k$-connected, no vertex cut of $G$ of size less than $k$ can separate $J$ from   $J'$. Moreover, both $J$ and $J'$ have cardinality at least $k$. So (a version of) Menger's theorem implies that there is a set of $k$ pairwise disjoint paths joining $J$ to $J'$. Hence, there is a matching, $T$, of cardinality $k$, with each edge of the matching joining a point in  $J$ to a point in $J'$.

Consider first an arbitrary complete graph $K_j$, of which  Lemma~\ref{l:decomp} can be used to obtain a full hamiltonian decomposition, together with a transversal containing one edge from each of the Hamilton cycles. 
Now choose $k/2$ of the Hamilton cycles in the decomposition, and let  $T'$ be the matching consisting of the edges of the transversal that lie in the chosen cycles.

Next, we can identify the set of vertices of   $K_j$ with the set $J$, such that the  vertices incident with edges in $T'$ are identified with  the vertices of $J$ that are  matched by $T$. From each of the Hamilton cycles, delete the edge, say $x$,  in that cycle that lies in $T'$, and add the two edges of $T$ adjacent to $x$. This gives a path $P$ in $G$ which starts and finishes in vertices in $J'$. Since $T'$ is a matching, the end vertices of all paths comprise a set of $k$ distinct vertices. Hence, these $k/2$ paths suffice for $F_1,\ldots, F_{k/2}$.
\medskip

 \noindent
{\em Case 2:   $j \le k$.}
\smallskip

Since $\GXrk$ is $k$-connected, each vertex in $J$ has degree at least $k$. The case $j=1$ is trivial, so we restrict our attention to $2\le j \le k$.  By Lemma~\ref{l:extradeg}, we may assume that one vertex in $J$, say $X$, has degree at least $k+1$. Hence, we have $d_H(X)\ge \ell+1$ where $\ell =k-j+1$, and $d_H(Y)\ge \ell$ for all other vertices $Y$ in $J$. Moreover, if any two vertices $X$ and $Y$ in $J$ had at most $\ell$ neighbours in $J'$ jointly, then those neighbours, together with the remaining $j-2$ members of $J$, would form a $(k-1)$-vertex cut, a contradiction. So the second condition in Lemma~\ref{l:colouring} is satisfied.
Thus, Lemma~\ref{l:colouring} ensures the existence of a colouring of (some of) the edges of $G$, such that each colour induces a linear forest in which each vertex in $J$ has degree $2$ and  whose ends are in $J'$. Since all edges in the paths of these forests are incident with $J$, at most $2j<2k$ vertices in $J'$ are used for each forest. So these forests can serve as the requisite linear forests $F_1,\ldots, F_{k/2}$. 
This completes the proof of   Claim~1.
\medskip

Let $J''$ be the set of vertices in $J'$ contained in some path of 
$F_1,\ldots, F_{k/2}$. By the claim above, we have $|J''|<k^2$. Moreover, by setting $r=r_l$ and $r'=r_k$ in Lemma~\ref{l:ring}, we deduce that each vertex in $J''$ is contained in a cell $c$ that is either dense or adjacent to some dense cell in $\cR$. (Note that $c$ may be either sparse or dense.) We extend each forest $F_i$ to a spanning forest $F'_i$ of $J\cup J''$ by adding those vertices in $J''$ not used by any path of $F_i$ as separate paths of length $0$.  The total number of paths is at most $|J''|k/2<k^3/2$, taking into account multiplicity since each $0$-length path may belong to more than one forest $F'_i$.

The next step is to associate each of the paths in $F'_i$ with the colour $i$, and create circuits $C_1,\ldots  C_{k/2}$ such that circuit $C_i$ contains $C$, together with an extra cycle $C(P)$ for each path $P$ in the forest $F'_i$.
To construct $C(P)$, take two cells $c$ and $d$ in $\cR$ (possibly $c=d$), one for each end-vertex of $P$, either containing the end-vertex  or adjacent to a cell containing it. Then the cycle $C(P)$ consists of a special new edge (possibly a loop) joining cells $c$ and $d$ (we say this edge {\em represents} $P$), together with a path of cells within $\cR $ joining those same cells $c$ and $d$. Note that each cycle has length bounded above by $10/\delta^2$ (see Lemma~\ref{l:ring}).

This construction was all with respect to a particular bad component $b$. Now repeat the construction for all the other bad components,  in each case extending the circuits $C_1,\ldots  C_{k/2}$ in the same way as for $b$. By Lemma~\ref{l:ring}, two paths $P$ and $P'$ related to different bad components have no vertex in common, and also the corresponding extra cycles $C(P)$ and $C(P')$ use disjoint sets of dense cells.
Hence, the number of these new cycles passing through any particular dense cell is at most $k^3/2$, so the maximum degree of dense cells in each $C_i$ is at most $2\Delta+k^3$. 

We next perform a version of Step~1 described in Part 1. First, let us call  all the  vertices lying in forests $F_i$ with respect to   bad components  the {\em forest} vertices.
\medskip

\noindent
{\em Step 1'. Turning  the circuits $C_i$ into  edge-disjoint coloured cycles}
\smallskip

The first part of this is done just as for Step~1 when $\cB$ empty, except for two aspects. To start with, all forest vertices within dense cells are set aside and not used in construction of the coloured cycles within the dense cells.  Secondly, where an edge of the circuit between cells $c$ and $c'$ is one of those representing a path $P$ in a forest, instead of a simple edge between vertices  $u$ in $c$ and $v$ in $c'$, the cycle uses the path  $P$ represented by that edge, together with the edges joining $P$ to $u$ and $v$. In this way, we obtain from $C_i$ a cycle of colour $i$ that  visits precisely all vertices that are either forest vertices (i.e.\  in $J(b)$ or $J''(b)$ for any small component $b$), or   in a dense cells, but {\em not} both.  All other points in $\bX$ will be called {\em outsiders}. They are not yet visited by any of  $C_1,\ldots,C_{k/2}$, either because they are neither forest vertices nor in a dense cell, or are in both.  

 We next perform a version of Step 2, as follows. 
\medskip

\noindent
{\em Step 2'. Extending the cycles to the outsiders}
\smallskip

This consists of extending each coloured cycle   as done in Step 2, but this time extending to them  only through the outsiders.  The other significant difference  between this and Step~2 is that the maximum degree of dense cells in each $C_i$ is now bounded above  by $2\Delta+k^3$ rather than $2\Delta$, and there are some outsiders in each dense cell, so there are fewer spare edges to work with, but the change is only a constant. So we     need to adjust the lower bound on $M$ accordingly. This completes the proof of part (i) of the theorem.
\medskip

\noindent
(ii) {\bf Proof for $k$  odd.}
\smallskip

We now consider odd $k\ge 1$. Recall that the total number of vertices in the geometric graph must be even. The same framework of argument is used as for $k$ even. For $k$ odd, colours $1,\ldots (k-1)/2$ will be used for the Hamilton cycles, and colour $(k+1)/2$ will be used for the matching. We find it convenient to refer to $(k+1)/2$ as the {\em match} colour. The edges that we colour using the match colour will form a spanning subgraph $D$ of the geometric graph, whose edges form a cycle on some (possibly all) vertices and a matching of some of the other vertices; in the very last stage of the argument we will adjust this to form a perfect matching of the whole graph.  

Define the multigraph $F$ as for $k$ even, and construct $C$ in the same way. 
We next need to perform a version of Step 1. In this case, instead of creating $k/2$ edge-disjoint cycles passing through all the points in dense cells, we construct $(k+1)/2$ of them using the same construction as for $k$ even. 

In the case that there are no bad cells, the argument as in Part~1 above shows that all the cycles can be extended as in Step~2 to edge-disjoint Hamilton cycles of the graph. Then, since the graph has an even number of vertices, every second edge of the matching colour can be omitted to provide the desired colouring.

So consider the case  that $\cB$ is possibly nonempty and follow the argument in Part~2   for $k$ even, up to the point where Claim~1  is made. This claim is replaced by the following.
\medskip

\noindent 
{\em Claim 2.\ }
$G$ contains $(k+1)/2$ pairwise edge-disjoint linear forests $F_1,\ldots, F_{(k+1)/2}$, such that in each of the first $(k-1)/2$  forests

\noindent
(a) all vertices in $J$ have degree $2$, and

\noindent
(b) at most $2k$ vertices in $J'$ are contained in any path,

\noindent
and the last forest, $F_{(k+1)/2}$, is a matching that saturates each vertex in $J$.

\smallskip

To prove this claim, we again distinguish two main cases but insert an extra one. 
\medskip

\noindent
{\em Case 1a:   $j>k+1$.}
\smallskip

As in the case $j>k$ for the first claim, we may find the matching $T$ of (odd) cardinality $k$, and also $(k+1)/2$ edge-disjoint Hamilton cycles in $G[J]$, with a matching $T'$ containing one edge from each Hamilton cycle. By relabelling $J$, we can then align the end vertices of the $T'$ with those of $T$ in $J$, but only using one end vertex of the edge in the last cycle. A set of edges in this last cycle is easily deleted so that what remains, together with the incident edge of $T$, is a matching either entirely contained in $J$ or with one edge (of $T$) leaving $J$. This proves the claim in this case.
\smallskip

\noindent
{\em Case 1b:   $j=k+1$.}
\smallskip

In this case, we may use the decomposition of a $K_{j+1}=K_{k+2}$   into $(k+1)/2$ Hamilton cycles, delete any one vertex, and use the broken Hamilton cycles for the paths, plus every second edge of any one of them for the   matching. 
\smallskip

\noindent
{\em Case 2:   $j\le k$.}
\smallskip

This is identical to the proof of Case 2 of the first claim, using the appropriate part of Lemma~\ref{l:colouring}.
This finishes the proof of  Claim~2.

We now extend the forests to spanning forests $F_i'$, exactly as for the even $k$ case (this now includes the match colour). We are now prepared for the analogue of Step~1.
\medskip

\noindent
{\em Step 1''. Circuits into cycles and matching}
\smallskip

For each non-match colour  $1\le i\le (k-1)/2$,  the forests $F_i'$ in the various bad components are   treated the same way as in Step~1' in order to create circuits $C_i$. These are then turned into edge-disjoint cycles $\widetilde C_i$ passing through all vertices except for the outsiders, just as for $k$ even. Simultaneously with this, by including an extra colour in the construction, we create a cycle in the match colour that passes through just the vertices in the dense cells apart from outsiders.  For this colour, we must do something different with respect to the vertices not in dense cells. So let $i=(k+1)/2$; so far we have a cycle $\widetilde C_i$, through all non-forest points in dense cells, whose edges are of colour $i$. The edges of a given forest $F_i$ related to a bad component are naturally coloured with $i$.  For each  single vertex component $v$ in the forest $F_{i}'$ (i.e.\ each vertex $v$ unmatched by $F_i$), let $d$ be a dense cell either equal or adjacent to the cell $c$ containing $v$. We may select the end-vertices $v_1$ and $v_2$ in $d$ of a spare edge of $\widetilde C_i$  (note this implies that the edges $vv_1$ and $vv_2$ are currently uncoloured).  Denoting these two vertices  $v_1$ and $v_2$, the pair $\{v_1,v_2\}$ is called the {\em gate} for $v$.  Each such vertex $v$ is treated in this way, its gate is defined and $v$ is   added to a set $W$ (a set of vertices which are  `waiting'). Naturally, any spare edge incident with either vertex in a gate is deemed non-spare for all subsequent choices of gates. Note that at this stage, all vertices in the graph are incident with an edge of the match colour except for the outsiders and those in $W$.
 
Next, we perform the step of extending the cycles to the outsiders.
\medskip

\noindent
{\em Step 2''. Extending cycles and matching to outsiders}
\smallskip

For each non-match colour $i$, the cycle of colour $i$ is extended as in Step~2'.  We next show that we can also include a near-perfect matching of the match colour, which saturates all but a bounded number of outsiders, such that the matching is edge-disjoint from all the coloured cycles. This extra matching is easily obtained using the methodology of Step~2': for the match colour $i$, if a  sparse cell $c$ contains at most $2k$ outsiders, we may leave all outsiders unmatched. If  it contains more outsiders, simply include an extra cycle $\widehat C$ through all outsiders in $c$. This cycle should be chosen simultaneously with all the other cycles being chosen within  cell $c$ in this Step~2'',  using Lemma~\ref{l:decomp}. Then, colour every second edge of $\widehat C$ with the match colour, leaving at most one outsider in this cell  unmatched. Finally, in either case, for each remaining unmatched outsider $v$, choose a gate $(v_1,v_2)$ exactly as described in Step~1''. Note that a dense cell contains a bounded number of outsiders since these are all in the forest $F_i$.

After all this, every vertex is in each of the coloured cycles of colour $i\le (k-1)/2$, but we still need to create the perfect matching of colour $i=(k+1)/2$. So far, all vertices are either matched by the match colour $i$, or lie in $W$, or lie on the cycle $\widetilde C_i$.  To fix this, in one fell swoop, we choose simultaneously for all vertices $v$ in $W$,  a vertex $v'$, in the gate for $v$,   such that all the vertices $v'$ have odd distance apart as measured along the cycle $\widetilde C_i$ of colour $i$. (Why this is possible will be explained shortly.) Then all such edges of the form $vv'$ are coloured $i$, and finally, every second edge along $C_i$ between these vertices is coloured $i$ in such a way as to create a matching. The edges of colour $i$ clearly form a perfect matching of the graph. 

The only thing left to explain is why the choice of all $v'$ as specified, creating odd distances, is feasible. Since there are two adjacent vertices on the cycle in each gate that can potentially be used as $v'$, we may pass along the cycle $\widetilde C_i$ making sure that the distances between chosen vertices are odd, until returning to the starting point. The very last distance must be odd because the number of gates equals the number of vertices in $W$. These are precisely the vertices outside the cycle that are not already matched by colour $i$. Since the number of vertices in the graph is even, the parity is correct for every distance to be odd.
\qed
\medskip

\noindent
{\bf Acknowledgment\ } We thank Alice and Millie for leaving the coloured pencils behind.

\begin{thebibliography}{99}
\bibitem{BaBoWa} J.~Balogh, B.~Bollob\' as, and M.~Walters.
\newblock Hamilton cycles in random geometric graphs.
\newblock Preprint, http://arxiv.org/abs/0905.4650

\bibitem{BoFr85} B.~Bollob\'as and A.~Frieze, On matchings and Hamiltonian cyclesin random graphs. In {\sl Random Graphs '83,\/} M. Karo{\'n}ski and A.Ruci\'nski (Eds.), North-Holland, Amsterdam, pp.\ 23--46 (1985).

\bibitem{DiMiPe07}
J.~D{\'\i}az, D.~Mitsche, and X.~P{\'e}rez.
\newblock Sharp threshold for hamiltonicity of random geometric graphs.
\newblock {\em SIAM Journal on Discrete Mathematics}, 21(1):57--65, 2007.

\bibitem{KrMu} M.~Krivelevich, and T.~M\" uller.
\newblock Hamiltonicity of the random geometric graph.
\newblock Preprint, http://arxiv.org/abs/0906.0071v2


\bibitem{Pe99} M.D.~Penrose.
\newblock On $k$-connectivity for a geometric random graph.
\newblock {\em Random Structures \& Algorithms}, 15(2):145--164, 1999.

\bibitem{Pe03}
M.D.~Penrose.
\newblock {\em Random Geometric Graphs}, volume~5 of {\em Oxford Studies in Probability}.
\newblock Oxford University Press, 2003.

\end{thebibliography}
\end{document}